\documentclass[11pt]{article}
\usepackage{amsthm}
\usepackage{geometry}
\usepackage{graphicx}
\usepackage{amsmath}
\usepackage{amssymb}
\usepackage{physics}
\usepackage{tikz-cd}
\usepackage{quiver}
\usepackage{setspace}
\usepackage{xcolor}
\usepackage{hyperref}
\usepackage{listings}
\usepackage{appendix}
\usepackage{mathtools}
\usepackage{placeins}
\usepackage{microtype}
\usepackage{enumitem}
\usepackage{doi}
\usepackage{url}
\usepackage[square,numbers]{natbib}
\setlist[enumerate]{leftmargin=.5in}
\setlist[itemize]{leftmargin=.5in}

\newtheorem{prop}{Proposition}
\newtheorem{thm}{Theorem}
\newtheorem{lemma}{Lemma}
\theoremstyle{definition}
\newtheorem{defn}{Definition}
\newtheorem{rem}{Remark}

\newtheorem{exa}{Example}
\date{\today}

\author{Michael R. Astwood \footnote{
       Department of Mathematics,
       University of Manitoba,
       Winnipeg, MB, R3T 2M6 Canada
       ({michael.astwood@umanitoba.ca})
    }
}

\title{Theoretical Aspects of Lie Groupoid and Lie Algebroid Equivariant Convolutional Neural Networks}

\begin{document}
\maketitle

\begin{abstract}
    We introduce Lie groupoid equivariant neural networks as a specialization of recently proposed topological category-equivariant neural networks to the differentiable setting. Lie groupoid equivariant neural networks are composed from Lie groupoid lifting convolutions and Lie groupoid convolution layers, and we show how for suitable Lie groupoids they are equivalent to certain Lie algebroid-equivariant neural networks. We additionally describe groupoid invariant global pooling as a generalization of group invariant global pooling. Furthermore, we show that each of the aforementioned layers is a special case of recently introduced admissible category-equivariant layers by demonstrating that they define continuous natural transformations between continuous feature functors. 
\end{abstract}

\section{Introduction}

A recent trend in deep learning has been to incorporate {\em a priori} known structures and symmetries into the design of a neural network rather than through dataset augmentation in order to improve inductive bias. Such ideas have brought forward group-equivariant neural networks \cite{pmlr-v48-cohenc16,cohen2019gauge,finzi2020}, graph neural networks \cite{Scarselli2009}, and various related combinations thereof \cite{pmlr-v139-satorras21a}. Most recently, this viewpoint has been generalized to category-equivariant neural networks \cite{maruyama}, which include group and groupoid-equivariant neural networks, poset-equivariant neural networks, and message passing graph neural networks as special cases. Many of these ideas fall under the broad topic of geometric deep learning (see e.g. the draft book by Bronstein et al. \cite{bronstein} for a classic reference). We specialize Maruyama's categorical equivariant neural networks to a special type of groupoid, namely Lie groupoids, for which there exists an infinitesimal analogue called a Lie algebroid that allows for a convenient parametrization in the appropriate circumstances. The structure of our paper is as follows.
\begin{enumerate}
    \item In the present section we describe the general motivation behind Lie groupoid and Lie algebroid equivariance in machine learning, and discuss some related topics in the literature.
    \item In section 2 we introduce Lie groupoids and Lie algebroids mathematically and provide several examples which may be of use to readers interested in implementing these ideas in practice.
    \item In section 3 we define our Lie groupoid and Lie algebroid convolutional layers, LieGrpdLiftingConv, LieGrpdConv, LieAlgbdLiftingConv, and LieAlgbdConv, as well as a groupoid-invariant global pooling operation we refer to as GrpdIGP.
    \item In section 4 we recall Maruyama's construction of category-equivariant neural networks \cite{maruyama} and prove that our Lie groupoid convolution layers are a special case of Maruyama's category convolutions.
\end{enumerate}

Groupoids are a type of higher generalization of a group, where `multiplication' of elements may only be a partial function. A groupoid consists of a set of objects, and a set of invertible morphisms between those objects. In the case where these sets are discrete spaces, one can construct the graph consisting of morphisms as edges and objects as vertices. A groupoid is able to represent local and global symmetries in both the continuous and discrete settings, in contrast to groups which encode global symmetries and graphs which encode local discrete connective structures. Lie groupoids appear when modeling systems with partial or local symmetries, such as symmetric systems with constraints or systems with stratifications. In general, applications of groupoid-equivariant neural networks have only just begun to emerge in the literature, with much of the general theory having been introduced in the recent PhD thesis of P. de Haan \cite{deHaan}, and the theory of Hausdorff topological groupoid equivariant neural networks having been introduced in the work of Maruyama \cite{maruyama}. More broadly, groupoids have been applied to control systems engineering and optimal control theory \cite{Arsie, Colombo2016} and reinforcement learning \cite{opperman2026groupoid} and in studying datasets with the structure of a fibration \cite{velarde2026fibration}. They have also been applied to information geometry in describing symmetries of contrast functions \cite{Grabowska2019}. A major source of examples of Lie groupoid symmetries in the sciences is in physics, where the gauge groupoid of a field theory encodes both global and local gauge transformations (see e.g. \cite{brown2002theorie}). We expect that Lie groupoid equivariant neural networks will find applications in datasets with continuous fibration symmetries, in the learning of contrast functions on statistical manifolds, as well as in lattice gauge theory. Indeed, Maruyama's categorical neural networks have been applied to gauge theoretic problems such as $\textrm{SO}(3)$ lattice gauge theory on a torus \cite{maruyama2026categorical}. 

From a topological point of view, topological (and Lie) groupoids can be seen as special simplicial complexes, where the simplices are generated by composable tuples of morphisms. Simplicial ideas in graph deep learning have been introduced by Bodnar et al. \cite{bodnar2021weisfeiler}. Examples of recent work on simplicial neural networks include that of Tang et al \cite{tang2025deepscnn} and Einizade et al. \cite{einizade2026continuous}. Our Lie groupoid neural networks provide a step towards generalizing simplicial deep learning to the continuous setting. This approach is quite different from the continuous simplicial neural networks of Einizade et al \cite{einizade2026continuous}, where instead simplicial neural networks are generalized to continuous data using partial differential equations. A more grand proposal is to extend groupoid-equivariant neural networks to higher groupoids (called $n$-groupoids), which are simplicial complexes satisfying certain conditions called Kan conditions (see e.g. \cite{Zhu2009}), for which one might coin the term `higher geometric deep learning'. A starting point may be the work of Rom\'an and Villatoro on convolution algebras of double Lie groupoids \cite{Roman2024}. Moreover, Maruyama has also recently proposed a method for $\infty$-categorical neural networks which uses the $E_\infty$-algebraic model of higher category theory, rather than a simplicial model \cite{Maruyama2026infty}.
This will be the subject of future research. 

We may also relate our work to the prior advent of group-equivariant and graph neural networks. Group-equivariant neural networks were introduced by Cohen \cite{pmlr-v48-cohenc16}. Let $G$ be a discrete group acting transitively and freely on a set $X$. A $G$-equivariant network is a convolutional neural network, taking inputs features over $X$, where conventional convolutions are replaced by convolutions on the group $G$, and feature maps are first lifted to functions on $G$ using a so-called lifting convolution. This has been extended to arbitrary group actions, including non-free and non-transitive actions as well as Lie groups \cite{finzi2020}, by considering lifts of features to the set $G\times X/\!/G$ rather than $G$, where $X/\!/G$ is the orbit space of $X$. For Lie groups it is often convenient to parametrize functions on $G$ in terms of the Lie algebra of $G$ using a local exponential map. The groupoid-equivariant neural networks developed in this article reduce to ordinary group-equivariant neural networks when the groupoid has only one object, and our Lie algebroid parametrization reduces to the ordinary Lie algebra parametrizations of Finzi et al \cite{finzi2020} and Dehmamy et al \cite{NEURIPS2021_148148d6}.

Graph neural networks are a similar kind of convolutional network, first formalized by Scarselli \cite{Scarselli2009}. Hybrid group-equivariant graph neural networks were later developed, which account for group-equivariance of features attached to each vertex of a graph (see e.g. \cite{satorras2021egnn,cohen2019gauge,rayat2,rayat,Favoni2022}). Graph neural networks primarily work by attaching features as labels on the edges and/or vertices of a graph, and then using a convolution related to the connective structure of the graph. For instance, a simple message-passing graph neural network with one dimensional feature maps has a kernel given by a weighted adjacency matrix. Meanwhile, diffusion-based models consider a message-passing scheme related to random walks on the graph. Since a finite groupoid can be interpreted as an undirected multigraph, one can view many of the simplest groupoid-equivariant neural networks as a special kind of graph neural network.

\section{Mathematical Preliminaries}
In this section we introduce the mathematical foundations of Lie groupoid-equivariant neural networks. We begin with the definition of a groupoid.
\begin{defn}[Groupoid]
    A groupoid $G_\bullet = (G_1\rightrightarrows G_0)$ consists of a set $G_1$ of morphisms (`arrows'), a set $G_0$ of objects (`points'), and set functions called the structure maps:
    \begin{itemize}
    \item A source map $s: G_1\to G_0$. 
    \item A target map $t: G_1\to G_0$.
    \item A composition map $m: G_2 \to G_1$, where \[
    G_2 = G_1\times_{s,t} G_1 = \{(h,g)\in G_1\times G_1 : s(h)=t(g)\}
    \]
    We often write $m(h,g)=h\cdot g$.
    \item A unit map $u:G_0\to G_1$, where $u(p)\cdot g=g$ and $h\cdot u(x)=h$ for all $g,h \in G_1, p\in G_0$ for which composition with $u(p)$ is defined. We often write $u(p)=1_p$
        \item An inverse map $i : G_1\to G_1$ with $i(g)\cdot g=u(s(g))$ and $g\cdot i(g)=u(t(g))$. We often write $i(g)=g^{-1}$.
    \end{itemize}
    The two arrows in the notation $G_1\rightrightarrows G_0$ correspond to the source and target map respectively. In other words, a groupoid is a small category where every morphism is invertible.
\end{defn}

\begin{defn}[Lie Groupoid]
    A Lie groupoid is a groupoid $G_\bullet$ where $G_0$ and $G_1$ are smooth manifolds, all structure maps are smooth, and $s$ and $t$ are surjective submersions. In particular, it is useful to visualize the unit map $u$ as a smooth embedding of $G_0$ into $G_1$. A diagram of this embedding is shown in Figure \ref{fig:groupoid}.
\end{defn}
\FloatBarrier
\begin{figure}[h]
    \centering
    \includegraphics[width=0.5\linewidth]{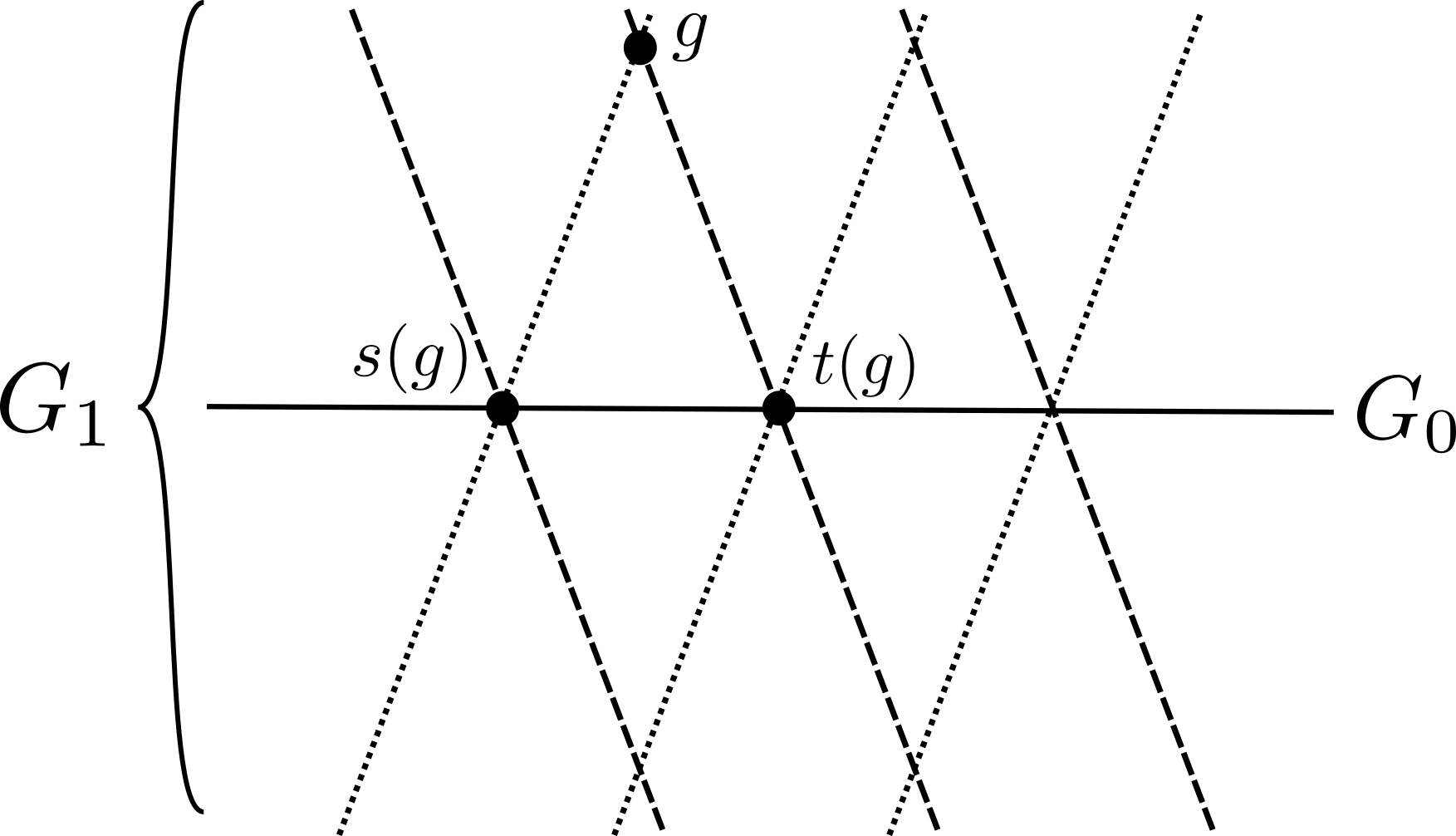}
    \caption{Diagram of a Lie groupoid $G_\bullet$, showing a point $g \in G_1$ along with its source and target, as well as source and target fibers $s^{-1}(x), t^{-1}(x)$ for several values of $x \in G_0$. Dashed lines correspond to source fibers, while dotted lines correspond to target fibers. The solid central line corresponds to $G_0$ embedded within $G_1$ via the unit map $u$.}
    \label{fig:groupoid}
\end{figure}

\FloatBarrier

\begin{exa}[Groupoid Defined by a Group] \label{def:delooping}
    A group $G$ can be used to define a groupoid $BG = G_1 \rightrightarrows G_0$ where $G_1=G$, and $G_0 = \{*\}$ is the trivial space. This groupoid is sometimes called the delooping of $G$, and it is intimately related to the classifying space of $G$ (often denoted by the same symbol $BG$ in homotopy theory).
\end{exa}
\begin{exa}[Strict Lie 2-Group]
    When $G_1$ and $G_0$ are both Lie groups and the structure maps $s,t,u,i$ are all Lie group homomorphisms, we say that $G_\bullet$ is a strict Lie 2-group. The most common, yet trivial, example of a strict Lie 2-group is $BG$ as defined in the preceding example.
\end{exa}
\begin{exa}[Pair Groupoid]
    Let $G_0$ be any set and let $G_1=G_0\times G_0 \rightrightarrows G_0$. Define the maps $s,t:G_1\to G_0$ by $s(a,b)=a$ and $t(a,b)=b$ respectively. Define $m((a,b),(b,c))=(a,c)$, $u(a)=(a,a)$, and $i(a,b)=(b,a)$. This construction is called the {pair groupoid} of the set $G_0$. The image of $u$ is called the {diagonal}.
\end{exa}

\begin{exa}[Fundamental Groupoid] \label{def:fundgroupoid}
    Let $M$ be a manifold and consider the set $\Pi_1(M)$ consisting of tuples $([\gamma],x,y)$, where $[\gamma]$ is a homotopy class of paths $\gamma : [0,1] \to M$ with fixed endpoints $x$ and $y$. This set forms a groupoid $\Pi_1(M)\rightrightarrows M$ called the fundamental groupoid, where composition is only defined between paths whose endpoints meet.
\end{exa}

\begin{exa}[Action Groupoid]
    Let $G$ be a group, with identity $e$, acting on a set $X$ by $x\mapsto g\cdot x$. The group action can be described using the {action groupoid} $G\times X \rightrightarrows X$, where $s(g,x)=x$, $t(g,x)=g\cdot x$, $m((g,h\cdot x)),(h,x))=(gh,x)$, $i(g,x)=(g^{-1},g\cdot x)$, and $u(x) = (e,x)$. 
\end{exa}

\begin{exa}[Gauge Groupoid] \label{def:gaugegroupoid}
    Let $M$ be a manifold and let $P$ be a principal $G$-bundle over $M$. The gauge groupoid $\mathcal{G}(P)_\bullet$ is the groupoid $(P\times P)/G \rightrightarrows M$ whose morphisms are equivalence classes of pairs $[p,p']$ where the equivalence relation is that $(p,p') \sim (pg,p'g)$ for all $g\in G$. One can show that such morphisms are equivalent to specifying $G$-equivariant maps between the fibers of $P$, i.e. maps $\varphi_{[p,p']} : P|_{\pi(p')} \to P|_{\pi(p)}$. 
\end{exa}

\begin{exa}[Groupoids Associated to Foliations] Let $M$ be a manifold and let $\mathcal{F}$ be a regular foliation of $M$. Consider the groupoid consisting of tuples $([\gamma],x,y)$, where $x,y$ lie in the same leaf of $\mathcal{F}$, and $\gamma$ is an equivalence class of paths from $x$ to $y$ consisting of all homotopic paths all lying entirely within that leaf. We call the resulting groupoid the monodromy groupoid $\textrm{Mon}(M,\mathcal{F}) \rightrightarrows M$.

Meanwhile, another groupoid can be constructed from a foliation, called the holonomy groupoid $\textrm{Hol}(M,\mathcal{F}) \rightrightarrows M$. The morphisms of this groupoid are germs of local diffeomorphisms between transversals of the foliation $\mathcal{F}$ at points $x$ and $y$. In other words, they measure the foliation holonomy (not to be confused with the holonomy of a connection) of paths along leaves in the manifold.
\end{exa}

\begin{defn}[Groupoid Action]
    Let $X$ be a set. A (left-)$G_\bullet$ {action} on $X$ consists of a map $J : X\to G_0$ called the anchor, and a map $a:G_1\times_{J,s} X \to X$ called the action, where \[G_1\ltimes X:=G_1\times_{J,s}X= \{(g,x) : s(g)=J(x)\}\]
    and $J(a(g,x)) = t(g)$. We denote $g\cdot x=a(g,x)$. 

    Similarly to the case of a group action, a groupoid action of $G_\bullet$ on $X$ produces an action groupoid that we denote by $G_1\ltimes X \rightrightarrows X$. We call a space $X$ equipped with such an action a $G_\bullet${-space}.
\end{defn}

\begin{defn}[Orbit and Stabilizer] Let $X$ be a set with a $G_\bullet$ action and let $x \in X$. The orbit of $x$ is defined as the set $[x]= \{y \in X : \exists g \textrm{ s.t. } g\cdot x = y\}$. The stabilizer of $x$ is defined as the set $G_1^x = \{g \in G_1 : s(g) = J(x), g\cdot x = x\}$. The set of all orbits is denoted by $X/\!/G_1$. 
\end{defn}
\begin{defn}[Isotropy Subgroupoid]
    When $G_\bullet$ is interpreted as acting on itself via the anchor $J=s$ (or $J=t$ depending on the convention), the orbits are called the orbits of $G_\bullet$ and each stabilizer $G_1^x$ forms a Lie group. These groups are called the isotropy groups of $G_1$. The collection of isotropy groups of $G_1$ can be collected into a manifold $\textrm{Iso}(G) = \bigsqcup_{x\in G_0} G_1^x$ having the structure of a Lie subgroupoid of $G_1$ over $G_0$ called the isotropy groupoid.
\end{defn}
\begin{rem}
    When $G_\bullet$ is a Lie groupoid, it is very rarely the case that $X/\!/G_1$ admits the structure of a manifold. In the case where the source and target maps are proper (i.e. the preimage of a compact set is compact), it can be seen as a stratified differentiable space. When it is \'etale (i.e. when $s$ and $t$ are local diffeomorphisms), the orbit space is the leaf space of a foliation. Finally, when the action groupoid is proper and \'etale, the orbit space is an orbifold. In the present work we will simply equip $X/\!/G_1$ with the discrete topology, which is entirely sufficient for our purposes.
\end{rem}

\begin{defn}[Equivariant Function]
    Let $X$ and $Y$ be two sets equipped with $G_\bullet$ actions. Then a function $f : X\to Y$ is said to be $G_\bullet${-equivariant} if $f(g\cdot x) = g\cdot f(x)$ for all $x \in X$.

    Note that this definition reduces to the usual one for group actions when $G_\bullet$ is replaced by a group.
\end{defn}

\begin{defn}[Haar System] A Haar system of measures on $G_1$ is a family of measures $\{\dd \lambda^p : p \in G_0\}$, with support $t^{-1}(p)$ for each $p$, such that
\begin{align*}
    p \mapsto \int_{G_1} f(g)\;\dd \lambda^p(g) 
\end{align*}
is continuous for all measurable $f$, and if $s(g')=t(g)$ then
\begin{align*}
    \int_{G_1} f(g)\;\dd\lambda^{s(g')}(g) = \int_{G_1} f(g'\cdot g)\;\dd \lambda^{t(g')}(g)
\end{align*}
\end{defn}
In addition to a system of measures on $G_1$, we will also assume for the remainder of this paper that any $G_\bullet$-space $X$ is equipped with a $G_\bullet$-invariant Radon measure $\dd\mu(x)$, for which there exists a compatible (possibly ill-behaved) measure $\dd\mu([x])$ on the quotient $X/\!/G_1$. Here, by compatible we mean that $\dd\mu(x')$ disintegrates (see e.g. the book by Fremlin \cite{Fremlin2008-bg} for a reference on the existence of disintegrations) into a product measure $\dd\lambda^{t(g)}(g)\dd\mu([x])$ where $\rho(x') = t(g)$. While this assumption is quite strong, in practice any such space will be discretized for numerical purposes, allowing us to ignore certain measure theoretic issues.

We also recall the definition of the infinitesimal analogue of a Lie groupoid, called a Lie algebroid. 
\begin{defn}[Lie Algebroid] A Lie algebroid $A$ over a manifold $M$ is a vector bundle $A \to M$ equipped with the additional data of a Lie bracket $[-,-]_A : \Lambda^2\Gamma(A) \to \Gamma(A)$ and a bundle map $\rho : A\to TM$ called the anchor map, such that $[\rho(a),\rho(b)]_{TM}=\rho([a,b]_A)$. 
\end{defn}
\begin{exa}[Tangent Lie Algebroid]
    The tangent bundle $TM$ equipped with the trivial anchor map $\rho = 1$ is a Lie algebroid.
\end{exa}

\begin{exa}[Lie Algebroid of a Lie Groupoid]
    Let $G_\bullet=(G_1\rightrightarrows G_0)$ be a Lie groupoid. The Lie algebroid of $G_\bullet$ is the vector bundle $A$ defined by $A=\ker\!\left.\left(\dd t\right)\right|_{G_0}$, equipped with a map $\rho=\dd s|_A : A\to TG_0\;$ called the anchor, and a Lie bracket $\,[\cdot,\cdot]_A$ on the sections of $A$ induced by the Lie bracket of right-invariant vector fields on $G_1$.
\end{exa}

Every Lie groupoid gives rise to a Lie algebroid in this way. However, unlike Lie algebras, not every Lie algebroid gives rise to a Lie groupoid. Lie algebroids which do give rise to a Lie groupoid are said to be integrable. Moreover, the same integrable Lie algebroid may result from multiple different Lie groupoids.

\begin{exa}[Action Algebroid]
     The Lie algebroid of an action groupoid $G\times M \rightrightarrows M$ is the trivial vector bundle $\mathfrak{g}\times M \to M$, where $\mathfrak{g} = T_e G$ is the Lie algebra of $G$. The anchor map is given by $\rho(\xi,x) = \xi^{\#}|_x$, where $\xi^\#$ is the infinitesimal symmetry generated by $\xi$, and the bracket is given by $[a(x),b(x) ] = [a(x),b(x)]_{\mathfrak{g}} - \mathcal{L}_{\rho(a(x))}b(x) + \mathcal{L}_{\rho(b(x))}a(x)$, where $\mathcal{L}$ denotes the Lie derivative and $[a(x),b(x)]_{\mathfrak{g}}$ denotes the Lie bracket of the $\mathfrak{g}$-valued coefficients of $a$ and $b$.
\end{exa}

\begin{exa}[Flow Algebroid and Flow Groupoid]
    Let $\vec{x}$ be any vector field on a manifold $M$ and suppose that the maximal domain of $\exp(t\vec{x})$ is $U\subseteq \mathbb{R}$ (with equality when $\vec{x}$ is a complete vector field). The flow algebroid of $\vec{x}$ is the trivial vector bundle $\mathbb{R}\times M \rightrightarrows M$ with anchor map $\rho(\lambda,p) = \lambda \vec{x}|_p$ and trivial bracket. The flow algebroid of a vector field is integrable, and integrates to a Lie groupoid $\mathcal{D}(\vec{x})\rightrightarrows M$ called the flow groupoid of $\vec{x}$, whose morphisms consist of tuples $(\Phi_{\vec{x}}(t),p), t\in U,p\in M$ where $\Phi_{\vec{x}}$ is the flow operator of $\vec{x}$.
\end{exa}

\begin{exa}[Atiyah/Gauge Algebroid]
The Lie algebroid of a gauge groupoid $(P\times P)/G \rightrightarrows M$ is the vector bundle $\textrm{at}(P)=TP/G \to M$, sometimes called the Atiyah algebroid, where the equivalence classes are defined by the standard action of $G$ on $TP$ by the differential of the action map. The anchor map is given by the differential of the bundle projection map $\pi : P\to M$, which is equivariant by definition and so descends to a map from $TP/G$ to $TM$. The bracket is simply induced by the ordinary Lie bracket of vector fields on $TP$.
\end{exa}

\begin{exa}[Foliation Algebroid] The tangent distribution $T\mathcal{F}$ to a regular foliation $\mathcal{F}$ of a manifold $M$ forms a vector subbundle of $TM$ which is closed with respect to the Lie bracket of vector fields. Therefore the tangent Lie algebroid structure on $TM$ restricts to $T\mathcal{F}$ giving rise to the foliation Lie algebroid of $\mathcal{F}$. Both the monodromy groupoid $\textrm{Mon}(M,\mathcal{F})$ and the holonomy groupoid $\textrm{Hol}(M,\mathcal{F})$ integrate the foliation algebroid.    
\end{exa}

More generally, we can define equivariance in terms of natural transformations. If we view a $G_\bullet$-space $X$ as being defined by the action groupoid $X_\bullet = (G\times_{s,\rho} X\rightrightarrows X)$, then the correct notion of a feature map should be a functor $F_X : X_\bullet \to \textsf{Vect}$, and an equivariant neural network $N$ should be a natural transformation of functors $N:F_X \Rightarrow F_Y$ for two $G_\bullet$-spaces $X,Y$. This is the approach introduced by de Haan \cite{deHaan} and generalized further by Maruyama \cite{maruyama}. We will review this in section \ref{sec:category}.

\section{Groupoid-Equivariant Convolutional Networks}
In this section we will propose various types of groupoid equivariant layers. These models are based on the standard definition of groupoid convolutions with respect to Haar systems of measures \cite{Mackenzie2005, williamsHaarSystemsEquivalent2016}. We are interested in modeling features taking values in a $G_\bullet$-space $X$. In particular, for the present work we will consider continuous features assigning real $d$-dimensional vectors to each point of $X$.

\begin{defn}[Feature Map]
Let $X$ be a smooth manifold. A feature map of $X$ is an integrable function $f \in \mathcal{L}^1(X,\mathbb{R}^d)$.
\end{defn}

As in the case of group-equivariant neural networks, the first layer lifts such a feature map to a function on the product $G_1\times X/\!/G_1$, where $X/\!/G_1$ is the orbit space of the action. The way this works is by first choosing representatives (`origins') $o_{[x]}$ of each orbit $[x]$ of the action groupoid $G_1\ltimes X$, and then mapping points $x \in X$ via the embedding $X \hookrightarrow G_1\times X/\!/G_1$, $x \mapsto (g,[x])$, such that $g\cdot o_{[x]} = x$. In practice one may sample representatives of each orbit and average over the results so that the final calculation is invariant to the choice of representative. Following the lift, one may then consider the action of $G_\bullet$ on $G_1\times X/\!/G_1$ induced by the canonical action of $G_\bullet$ on $G_1$ by left multiplication, which lifts the action of $G_\bullet$ on $X$. This allows us to construct the following operations inspired by the constructions of Finzi et al \cite{finzi2020}.

\begin{defn}[LieGrpdLiftingConv]\label{def:LieGrpdLiftingConv}
Let $f$ be a feature map on a $G_\bullet$-space $X$ with values in $\mathbb{R}^{c_{\rm in}}$. A groupoid kernel $k$ is defined to be a compactly supported continuous function 
\[k : G_1\times (X/\!/G_1)^2 \to \mathbb{R}^{c_{\rm out}\times c_{\rm in}}\]
We define the lifting convolution by,
\begin{equation*}
    L_k \{f\}(g,[x]) = b([x]) + \int_{X/\!/G_1}\int_{t^{-1}(\rho(y))}k(\tilde{g}^{-1}g,[x],[y])f(\tilde{g}\cdot y)\;\dd\lambda^{\rho(y)}(\tilde{g})\dd\mu([y])
\end{equation*}
\end{defn}
where $b$ is any integrable function.
The above Lie groupoid lifting convolution reduces to the usual lifting convolution of Finzi et al. \cite{finzi2020} when $G_\bullet = G\rightrightarrows \{*\}$ is a Lie group. We next define convolution for functions $f : G_1\times X/\!/G_1 \to \mathbb{R}^{c_i}$.
\begin{defn}[LieGrpdConv]
\label{def:LieGrpdConv}
Let $f$ be a feature map on the lifted space $G_1\times X/\!/G_1$ with values in $\mathbb{R}^{c_{\rm in}}$. Then we again define a groupoid kernel to be a continuous compactly supported function,
\[k : G_1\times(X/\!/G_1)^2 \to \mathbb{R}^{c_{\rm out}\times c_{\rm in}}\]
and set
\begin{equation*}
    C_k\{f\}(g,[x]) = b([x])+\int_{X/\!/G_1}\int_{t^{-1}(\rho(y))}k(\tilde{g}^{-1}g,[x],[y])f(\tilde{g}, [y])\;\dd\lambda^{\rho(y)}(\tilde{g})\dd\mu([y])
\end{equation*}
where $b$ is any integrable function.
\end{defn}
The only difference between this definition and the previous is that in a lifting convolution, we start with a function $f : X \to \mathbb{R}^{c_{\rm in}}$, while in the ordinary convolution, we start with a function $f : G_1\times X/\!/G_1 \to \mathbb{R}^{c_{\rm in}}$. 

\begin{prop}
The LieGrpdLiftingConv and LieGrpdConv operators are $G_\bullet$-equivariant.
\end{prop}
\begin{proof}
We treat LieGrpdLiftingConv since the proof for LieGrpdConv is identical with 
$f(\tilde{g}\cdot y)$ replaced by $f(\tilde{g},[y])$. Let $h\in G_1$ with 
$t(h)=t(g)$. The proof will follow nearly immediately from the invariance properties of the Haar system. Expanding the definition,
\begin{align*}
L_k\{f\circ (h^{-1}\cdot)\}(g,[x])
&= \int_{X/\!/G_1}\int_{t^{-1}(\rho(y))} 
k(\tilde{g}^{-1}g,[x],[y])\,
f\!\left(h^{-1}\cdot(\tilde{g}\cdot y)\right)
\dd\lambda^{\rho(y)}(\tilde{g})\,\dd\mu([y]).
\end{align*}
Substituting $\tilde{g}' = h^{-1}\tilde{g}$, so $\tilde{g}^{-1}g = \tilde{g}'^{-1}(h^{-1}g)$, and applying left-invariance of the Haar system,
\begin{align*}
&= \int_{X/\!/G_1}\int_{t^{-1}(\rho(y))} 
k(\tilde{g}'^{-1}(h^{-1}g),[x],[y])\,
f\!\left(\tilde{g}'\cdot y\right)
\dd\lambda^{\rho(y)}(\tilde{g}')\,\dd\mu([y])
= L_k\{f\}(h^{-1}g,[x]).
\end{align*}
Hence $L_k\{f\circ(h^{-1}\cdot)\}(g,[x]) = L_k\{f\}(h^{-1}g,[x])$, which is what we wanted to show.
\end{proof}
\begin{rem}
    The substitution $\tilde{g}' = h^{-1}\tilde{g}$ is valid within each fiber $t^{-1}(\rho(y))$ provided $t(h)=\rho(y)$. Since the kernel $k(\tilde{g}^{-1}g,[x],[y])$ depends on the orbit labels $[x],[y]$ only as parameters and the $G_\bullet$-action leaves these fixed, we may restrict without loss of generality to morphisms $h$ satisfying $t(h)=\rho(y)$ for each $[y]\in X/\!/G_1$. This is always possible when $G_\bullet$ acts transitively on $G_0$, and in general one should interpret the equivariance property as holding fiber-by-fiber over the orbit space $X/\!/G_1$.
\end{rem}

We have now constructed the initial and intermediate layers of a groupoid-equivariant convolutional network. The final layer of a typical deep convolutional neural network is a global pooling layer. This must be constructed in an invariant way. 

Group Invariant Global Pooling (GIGP), introduced by Bujel et al \cite{bujel2023}, uses integration over the orbits as a way to pool feature maps on Lie groups. For a Lie group $G$, a $G$-space $X$, and a compactly supported function $f : X \to \mathbb{R}^m$, the GIGP function is $\textrm{GIGP}(f)([x]) = \int_{G} f(g\cdot x) \,\dd \lambda(g), [x] \in X/\!/G$, where $\lambda$ is the normalized Haar measure. For use with groupoid equivariant networks such as $G_\bullet$-conv and LieGrpdConv we therefore introduce Groupoid Invariant Global Pooling.
\begin{defn}[GrpdIGP]\label{def:GrpdIGP}
    Let $f$ be a feature map on the lifted space $G_1\times X/\!/G_1$ with values in $\mathbb{R}^{c_{\rm in}}$. Then we define the groupoid-invariant global pooling of $f$ as, 
    \begin{equation*}
        \textrm{GrpdIGP}\{f\}([x]) =b([x])+ \int_{t^{-1}(\rho(x))} f(g, [x])\,\dd\lambda^{\rho(x)}(g)
    \end{equation*}
\end{defn}
where $b$ is any integrable function on the orbit space $X/\!/G_1$.
Since GrpdIGP is $G_\bullet$-invariant, it follows that any remaining layers in the neural network will also be invariant.
\begin{rem}
    A transitive groupoid is a groupoid with only one orbit. In the case where the action groupoid $G_1\ltimes X \rightrightarrows X$ is transitive, the GrpdIGP operation reduces to a simple mean-reduction global pooling operation.
\end{rem}
To summarize, a diagram of the proposed neural network architecture is provided below in Figure \ref{fig:architecture}. Pointwise nonlinearities (activations) are to be inserted between each layer.
\FloatBarrier
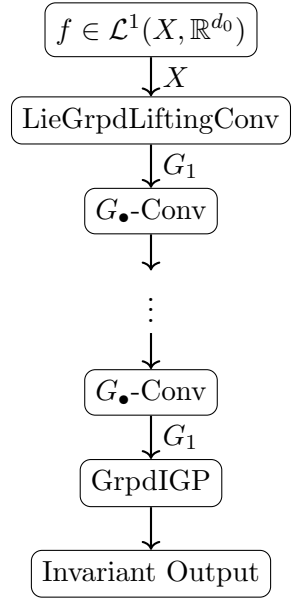
\begin{figure}[h]
\centering
\begin{tikzpicture}[
    node distance=1.2cm,
    every node/.style={
        draw,
        rectangle,
        rounded corners,
        align=center,
        inner sep=4pt
    },
    arrow/.style={->, thick}
]

\node (input) {$f \in \mathcal{L}^1(X,\mathbb{R}^{d_0})$};

\node (lift) [below of=input] {LieGrpdLiftingConv};

\node (gconv1) [below of=lift] {$G_\bullet$-Conv};

\node (dots) [below of=gconv1, draw=none] {$\vdots$};

\node (gconvN) [below of=dots] {$G_\bullet$-Conv};

\node (pool) [below of=gconvN] {GrpdIGP};

\node (output) [below of=pool] {Invariant Output};

\draw[arrow] (input) -- node[right, draw=none] {$X$} (lift);

\draw[arrow] (lift) -- node[right, draw=none] {$G_1$} (gconv1);

\draw[arrow] (gconv1) -- (dots);
\draw[arrow] (dots) -- (gconvN);

\draw[arrow] (gconvN) -- node[right, draw=none] {$G_1$} (pool);

\draw[arrow] (pool) -- node[right, draw=none] {} (output);

\end{tikzpicture}

\caption{Schematic of a Lie groupoid invariant neural network architecture, not including activations. An equivariant architecture can be obtained by removing the final pooling layer.}
\label{fig:architecture}
\end{figure}

\begin{rem}
    We here comment on the implementation of Lie groupoid-equivariant neural networks. While Lie groups consist only of a single manifold $G$, a Lie groupoid consists of morphisms between pairs of points in a manifold $M$. Thus, at worst case one can expect the size of an action groupoid (and hence the number of morphisms used in our construction) to scale quadratically with the size of the space $M$ upon which a groupoid acts. 
\end{rem}

\FloatBarrier

\section{Lie Algebroid Parametrization}\label{sec:algebroid_kernel}

Let $G_\bullet=(G_1\rightrightarrows G_0)$ be a Lie groupoid with Lie algebroid $A=\ker\!\left.\left(\dd t\right)\right|_{G_0}$ having anchor map $\rho=\dd s|_A : A\to TG_0$. For each $p\in G_0$ there exists an open precompact neighbourhood $U_p\subseteq A_p$ and a smooth groupoid exponential map
\[
\exp_p : U_p \to s^{-1}(p), \quad p\in G_0
\]
such that $\exp_p(0_p)=u(p)$, and $\exp_p$ is a diffeomorphism onto a neighbourhood $V_p$ of $u(p)$ in $G_1$ (see \cite[Proposition 3.6.1]{Mackenzie2005}, \cite[Theorem 3.2]{kubarski1982exponential}, and references within).

Fix a smooth positive density $\dd \nu_p$ on each fiber $A_p$. Let $k : G_1\times (X/\!/G_1)^2 \to \mathbb{R}^{c_{\rm out}\times c_{\rm in}}$ be a kernel function. Then define
\[K(v,[x],[y]) := k(\exp(v),[x],[y])\]
and assume that $K$ is compactly supported on a compact subset $\overline{U}_p$ (the closure of $U_p$ above) of each fiber $A_p$ of $A$, such that in particular, $\exp$ is a local diffeomorphism on $\overline{U}_p$ (i.e. it has a smooth inverse $\log(g)$ for all $g$ in the support of $k$). Let $J_p(v)$ be the Jacobian determinant of $\exp_p$ with respect to $\nu^p$ and the Haar system $\lambda^p$.  
\begin{defn}[LieAlgbdLiftingConv]\label{def:LieAlgbdLiftingConv}
    Given a feature map $f : X \to \mathbb{R}^{\rm c_{\rm in}}$, we define the LieAlgbdLiftingConv of $f$ with respect to the kernel $k$ by,
    \begin{equation*}
    \begin{split}
        \tilde{L}_k\{f\}(v,[x]) &:= 
        \int\limits_{X/\!/G_1}\int\limits_{A_{\rho(x)}}
            \left[K(v-w,[x],[y])\; f\!\left(\exp(w)\cdot y\right) J_{\rho(x)}(w)\right]\,\dd\nu^{\rho(x)}(w) \dd\mu([y]).
        \end{split}
    \end{equation*}
    to which we may add any integrable bias $b([x])$.
\end{defn}
\begin{defn}[LieAlgbdConv]\label{def:LieAlgbdConv}
    Given a feature map $f : G_1\times X/\!/G_1 \to \mathbb{R}^{\rm c_{\rm in}}$, we define the LieAlgbdConv of $f$ with respect to the kernel $k$ by,
    \begin{equation*}
    \begin{split}
        \tilde{L}_k\{f\}(v,[x]) &:= 
        \int\limits_{X/\!/G_1}\int\limits_{A_{\rho(x)}}
            \left[K(v-w,[x],[y])\; f\!\left(\exp(w), [y]\right)  J_{\rho(x)}(w)\right]\,\dd\nu^{\rho(x)}(w) \dd\mu([y]).
        \end{split}
    \end{equation*}
    to which we may add any integrable bias $b([x])$.
\end{defn}
For convenience in stating the coming proposition we also define the Lie algebroid convolution binary operation $(-)*_A(-)$ as,
\[f *_A K(v,[x]) = \tilde{L}_k\{f\}(v,[x]).\]
The LieAlgbdConv and LieAlgbdLiftingConv layers yield equivariant functions with respect to $G_1$ provided that the kernel $K$ is $\textrm{Ad}$-equivariant with respect to $G_\bullet$, in that $K(\xi,[x],[y])=K(\mathrm{Ad}_{g^{-1}}(\xi),[x],[y])$ for all $[x],[y]\in X/\!/G_1$ and $\xi \in A$. This is summarized by the following proposition
\begin{prop}\label{thm:algebroid_equiv}
Let $f\in \mathcal{L}^1(X,\mathbb{R}^m)$ and let $K$ be a smooth, compactly-supported kernel as above. Suppose $K$ is $\textrm{Ad}$-equivariant. Then for any $g\in G_1$, any $x\in X$ with $s(g)=\rho(x)$, and any $v\in A_{\rho(x)}$,
\begin{equation*}
    \left(f\circ (g^{-1}\cdot) *_{A} K\right)\left(\mathrm{Ad}_g(v),\,[g\cdot x]\right)
    \;=\;
    \left(f*_{A} K\right)\!(v,[x]).
\end{equation*}
Hence the LieAlgbd(Lifting)Conv operator $f\mapsto \tilde{L}_k\{f\}$ is $G_\bullet$-equivariant.
\end{prop}
\begin{proof}
Let $v\in A_{\rho(x)}=A_{s(g)}$, so that $\mathrm{Ad}_g(v)\in A_{t(g)}=A_{\rho(g\cdot x)}$. Expanding the definition of the convolution of $f\circ(g^{-1}\cdot)$ with $K$:
\begin{align*}
&\left(f\circ(g^{-1}\cdot)*_{A}K\right)\left(\mathrm{Ad}_g(v),[g\cdot x]\right)\\
&\quad= \int\limits_{X/\!/G_1}\;\;\int\limits_{\mathclap{A_{t(g)}}} K\!\left(\mathrm{Ad}_g(v)-w,[g\cdot x],[y]\right)
f\!\left(g^{-1}\cdot\left(\exp_{t(g)}(w)\cdot(g\cdot y)\right)\right)
J_{t(g)}(w)\;\dd\nu^{t(g)}(w)\dd\mu([y]).
\end{align*}

Since $[g\cdot x]=[x]$ in $X/\!/G_1$ and $K$ is equivariant (so $K(\xi,[x],[y])=K(\mathrm{Ad}_{g^{-1}}(\xi),[x],[y])$), and noting that $\mathrm{Ad}_{g^{-1}}(\mathrm{Ad}_g(v)-w)= v - \mathrm{Ad}_{g^{-1}}(w)$, this becomes
\begin{align*}
&= \int\limits_{X/\!/G_1}\;\;\int\limits_{\mathclap{A_{t(g)}}} K\!\left(v - \mathrm{Ad}_{g^{-1}}(w),[x],[y]\right)
f\!\left(g^{-1}\cdot\left(\exp_{t(g)}(w)\cdot(g\cdot y)\right)\right)
J_{t(g)}(w)\dd\nu^{t(g)}(w)\dd\mu([y]).
\end{align*}

We now change variables via the groupoid adjoint action $\mathrm{Ad}_{g^{-1}}:A_{t(g)}\to A_{s(g)}$, setting $u=\mathrm{Ad}_{g^{-1}}(w)\in A_{s(g)}$. By invariance of the Haar system under the adjoint action,
\[
J_{t(g)}(w)\dd\nu^{t(g)}(w) = J_{s(g)}(u)\dd\nu^{s(g)}(u).
\]
We also have,
\[
\exp_{t(g)}(w) = \exp_{t(g)}(\mathrm{Ad}_g(u)) = g\exp_{s(g)}(u)g^{-1},
\]
so the argument of $f$ simplifies:
\begin{align*}
g^{-1}\cdot\left(\exp_{t(g)}(w)\cdot(g\cdot y)\right)
&= g^{-1}\cdot\left(g\exp_{s(g)}(u)g^{-1}\cdot g\cdot y\right)\\&
= g^{-1}\cdot g\cdot\left(\exp_{s(g)}(u)\cdot y\right)\\&
= \exp_{s(g)}(u)\cdot y.
\end{align*}
Substituting everything, we obtain
\begin{align*}
&\left(f\circ(g^{-1}\cdot)*_{A}K\right)\left(\mathrm{Ad}_g(v),[g\cdot x]\right)\\
&= \int\limits_{X/\!/G_1}\;\;\int\limits_{\mathclap{A_{s(g)}}} K\!\left(v-u,[x],[y]\right)
f\!\left(\exp_{s(g)}(u)\cdot y\right)
J_{s(g)}(u)\dd\nu^{s(g)}(u)\dd\mu([y])\\
&= \left(f*_{A}K\right)(v,[x]).
\end{align*}
Hence $f\mapsto\tilde{L}_k\{f\}$ is equivariant under the groupoid action.
\end{proof}

\begin{rem}[Relation to LieConv]
When $G_\bullet$ is a Lie group $G\rightrightarrows\{*\}$, the Lie algebroid $A$ is the Lie algebra $\mathfrak g$ considered as a trivial bundle over a point $\{*\}$, and the LieAlgbdConv layer reduces exactly to the LieConv kernel parameterisation of \cite{finzi2020} in the case where $G$ acts freely.  
Thus the above recovers LieConv as the special case $A=\mathfrak g \to \{*\}$. Additionally, when $G$ does not act freely, one can obtain the action algebroid $A = \mathfrak{g}\times X$, in which case we again recover the Lie algebra parametrization for the orbit-wise groupoid convolution.
\end{rem}

\begin{rem}[Implementation]
Numerically one samples $v\in U_p$ (e.g.\ via Monte Carlo or a fixed lattice in $A_p$) and approximates the integral over the fiber of $A$.  
Because $K$ has compact support by assumption, only a finite number of samples are needed for each $x$. A specialized sampling procedure such as Markov-Chain Monte Carlo may be needed to ensure sample points are distributed according to the measure on each fiber. In practice, a uniform measure is the most likely scenario, in which case a simple lattice will do the job. Back-propagation is straightforward since all maps are smooth.
\end{rem}

\section{Categorical Interpretation of Groupoid Convolutions}
\label{sec:category}
\subsection{Categorical Equivariant CNNs}

To begin, we will reintroduce the notion of a feature functor and a categorical kernel introduced by Maruyama, and then show how our proposed architecture fits into this framework. The following reviews the definitions of \cite[Section~2]{maruyama}.

Let $G_\bullet$ be a Lie groupoid and let $\mathsf{Vect}$ denote the category of vector spaces. We will first review the definition of a feature functor.

\begin{defn}[Feature Functor {\cite[Definition~2]{maruyama}}]
\label{def:feature-functor}
Assign to each $p \in G_0$ a compact space $\Omega(p)$ locally homeomorphic to $\mathbb{R}^{m_p}$, where $m : G_0\to \mathbb{N}$ is any function. Let $\mathcal{U} = \{\Omega(p)\}_{p \in G_0}$ denote the set of such associated spaces.

Next, let $E : G_0\to \mathsf{Vect}$ be a family of 
finite-dimensional real vector spaces associated to each point of $G_0$. We write $E_p$ for the image $E(p)$, and will call it the feature fiber above $p$. 

From $E$ and $\mathcal{U}$ we build two auxiliary categories: the 
{base category} $\mathcal{U}_\bullet$, with objects $\Omega(p)$ and 
morphisms given by continuous maps $\Omega(p) \to \Omega(q)$; and the {transport 
category} $E_\bullet$, with objects $p\in G_0$ and morphisms given by 
linear maps $E_p \to E_p$. This is reminiscent of a gauge groupoid, except there is no invertibility requirement for a morphism.

To each arrow $g\in G_1$ we assign:
\begin{itemize}
    \item a continuous {base transport} $\tau_{g} : \Omega(s(g)) \to \Omega(t(g))$, 
    with $\tilde{\tau}_{g} := \tau_{g^{-1}} : \Omega(t(g)) \to \Omega(s(g))$ 
    its reverse,
    \item a linear {fiber transport} $L_{g} : E_{s(g)} \to E_{t(g)}$.
\end{itemize}
These are required to define functors
\begin{align*}
    \tau &: G_\bullet \to \mathcal{U}_\bullet, 
    & p &\mapsto \Omega(p), & g &\mapsto \tau_{g}, \\
    \tilde{\tau} &: G_\bullet^{\mathrm{op}} \to \mathcal{U}_\bullet, 
    & p &\mapsto \Omega(p), & g &\mapsto \tilde{\tau}_{g}, \\
    L &: G_\bullet \to E_\bullet, 
    & p &\mapsto E_p, & g &\mapsto L_{g}.
\end{align*}
Note that mixed-base compatibility, assumed by Maruyama as a separate condition, follows from the assumption that $G_\bullet$ is a groupoid:
\[
    \tau_{g\circ h} \circ \tilde{\tau}_{h} 
    \;=\; \tau_{g}
\]
We also assume continuity of all 
assignments with respect to the compact-open topology on hom-sets. A feature functor associated to $E,\mathcal{U}$ is then a functor
\[
    F : G_\bullet^{\mathrm{op}} \to \mathsf{Vect}
\]
defined on objects by $F(p) = C^0(\Omega(p), E_p)$ and on arrows by
\[
    F(g) : F(t(g)) \to F(s(g)), \qquad 
    F(g)(f) \;=\; L_{g^{-1}} \circ f \circ \tau_{g}.
\]
\end{defn}

Functoriality of $F$ follows from functoriality of $\tau$ and $L$. Each $F(p)$ is a space of local 
continuous feature maps on $\Omega(p)$, taking values pointwise in the feature fiber 
$E_p$. 

\begin{rem}[Special Cases]From a differential geometry point of view, one can interpret $F$ as assigning to each point $p$ a trivial vector bundle $E_p\times \Omega(p)$ over $\Omega(p)$ with typical fiber $E_p$, whose sections are the features at $p$. A special case of this would be to take $\Omega(p) = \{p\}$, so that $\sqcup E_p$ has the structure of a vector bundle over $G_0$. 

Another special case would be to assign a vector space of a different dimension to each point $p$. A third special case would be to take each $\Omega(p)$ to be the closure of a small enough precompact open neighbourhood of each $p$, and take $\{E_p: p \in G_0\}$ to be a presheaf of vector spaces on the corresponding category of open sets. A schematic picture of possible feature functors is shown in Figure \ref{fig:varyingdimension}.

However, we remark that these special cases do not include any mixing between disconnected components of the underlying space $G_0$. In particular, in group-equivariant and groupoid-equivariant convolutional neural networks we have mixing between features over disjoint orbits $[x]$ and $[y]$. This will be accounted for in subsection \ref{subsec:GeneralCase}.
\end{rem}

\begin{figure}[h]
    \centering
    \includegraphics[width=0.8\linewidth]{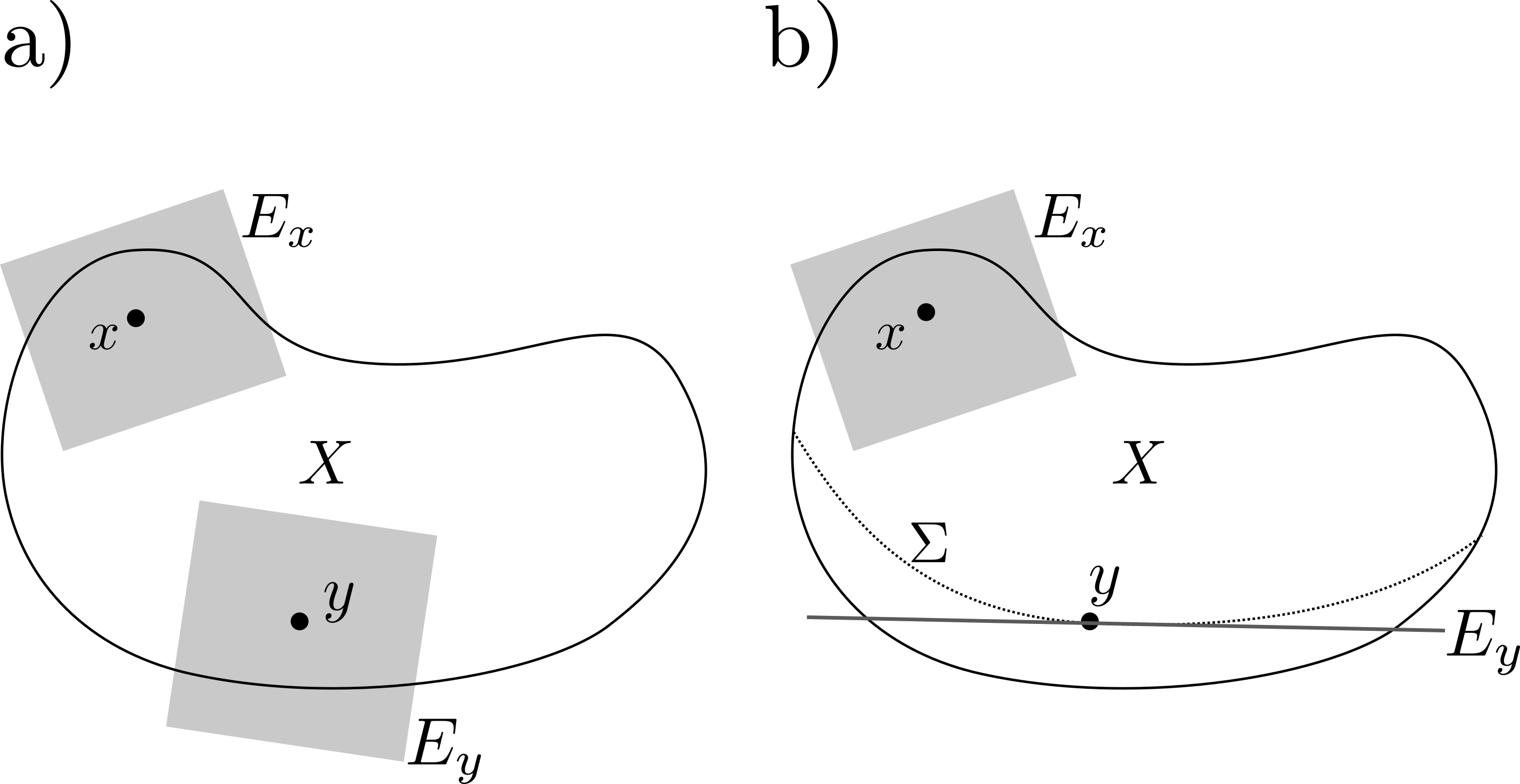}
    \caption{a) Diagram of a manifold $X = G_0$, showing the fibers $E_x$ of a vector bundle defining feature fibers of constant dimension $2$ over associated spaces $\Omega(x)=\{x\}$. b) Diagram of a stratified space $X$ with stratification $X = (X\setminus \Sigma) \sqcup \Sigma$, where the feature fibers are again defined over associated spaces $\Omega(x)=\{x\}$, and above $y \in \Sigma$ the fiber is 1-dimensional while the fiber above $x \in X\setminus\Sigma$ is 2-dimensional.}
    \label{fig:varyingdimension}
\end{figure}

\FloatBarrier

\begin{defn}[Continuous Equivariant Map {\cite{maruyama}}]
\label{def:equivariant-map}
Let $F, F' : G_\bullet^{\mathrm{op}} \to \mathsf{Vect}$ be feature functors associated 
to $E, E'$ respectively. A continuous equivariant map from $F$ to $F'$ 
is a natural transformation $\Phi : F \Rightarrow F'$ such that for each $p \in G_0$, the 
map $\Phi_p : F(p) \to F'(p)$ is continuous with respect to the sup-norm topology defined by the bases.
\end{defn}

Naturality of $\Phi$ means that for every arrow $g\in G_1$,
\begin{equation*}
\label{eq:naturality}
    F'(g) \circ \Phi_{t(g)} \;=\; \Phi_{s(g)}\circ F(g).
\end{equation*}
Unwinding the definitions, this says that for all $f \in F(g)$ and $q\in \Omega(s(g))$,
\[
    L_{g^{-1}}\left(\Phi_{t(g)}(f)(\tau_{g}(q))\right) 
    \;=\; \Phi_{s(g)}\left(L_{g^{-1}} \circ f \circ \tau_{g}\right)(q).
\]

\begin{defn}[Category Kernel {\cite[Definition~4]{maruyama}}]
\label{def:category-kernel}
Let $F, F'$ be feature functors associated to $E, E'$ as above. Let $L(E_q,E_p')$ be the space of linear maps from $E_q$ to $E_p'$. A category kernel is a family
\[
    K \;=\; \{K_{q \to p}\}_{(q,p) \in G_0^2}, 
    \qquad 
    K_{q\to p} : \mathrm{Hom}_{G_\bullet}(q,p) \times \Omega(p) 
    \;\longrightarrow\; L\left(E_a,\, E'_p\right),
\]
satisfying Carath\'eodory regularity, a uniform $L^1$ bound with respect to the Haar system, and the integrated naturality condition: for every arrow $h\in G_1$ and every $q\in \Omega(p)$,
\begin{equation*}
\label{eq:integrated-naturality}
\begin{split}
    &L_{h^{-1}} \int_{\mathrm{Hom}_{G_\bullet}(-,\, t(h))} 
    K_{r \to t(h)}(g,\, \tau_{h}(q))\, F(g)(-)\, d\lambda^{\rho(x)}(g)
    \;\\&=\;
    \int_{\mathrm{Hom}_{X_\bullet}(-,\, p)} 
    K_{r\to p}(g,\, q)\, F(h \circ g, r)(-)\, d\lambda^{\rho(p)}(g).
    \end{split}
\end{equation*}
\end{defn}
The category convolution associated to a kernel $K$ and a natural bias 
$b = \{b_p \in F'(p)\}$ is the family $\mathcal{L}_K = \{(\mathcal{L}_K)_p\}$ defined by
\begin{equation*}
\label{eq:category-convolution}
    \mathcal{L}_K \{f\}_p(q) 
    \;=\; b_p(q) \;+\; 
    \int_{I(p)} K_{r\to p}\left(g,\, q\right)\, 
    F(g)(f_r)\left(\tilde{\tau}_{g}(q)\right)\, 
    d\lambda^{\rho(p)}(g)\, d\mu(r),
\end{equation*}
where $I(p) = \bigsqcup_{r \in X} \mathrm{Hom}_{X_\bullet}(r,p)$ is the incoming arrow 
bundle at $p$. Note that $I(p) = t^{-1}(p)$ in the standard Lie groupoid notation. 

Note that $G_\bullet$ equipped with a Haar system is precisely a small topological category with measures in the sense of \cite[Definition~1]{maruyama}. 

\begin{defn}[Natural Bias]\label{def:naturalbias} Let $F : G_\bullet^{\rm op} \to \mathsf{Vect}$ be a feature functor with $E$ the associated family of vector spaces. A natural bias $b$ for $F$ is a family of functions $\{b_p \in C^0(\Omega(p),E_{p})$ such that for all $u : p \to q$ in $G_1$ we have $F(u)b_q = b_p$. 
\end{defn}

\subsection{Without Orbit-Mixing}
Let us consider now the special case where $\Omega(p) = \{p\}$, and where $E$ is equivalent to the data of a smooth 
real vector bundle over $G_0$ equipped with a left $G_\bullet$-action: that is, for each 
$g \in G_1$, a linear isomorphism 
$\hat{g} : E_{s(g)}\xrightarrow{\;\sim\;} E_{t(g)}$, satisfying $\widehat{hg} = \hat{h} \circ \hat{g}$ and $\widehat{1_{\rho(p)}} = \mathrm{id}_{E_p}$. Set $\varrho(g)=\hat{g}$. The second condition implies that the bundle projection $\pi : E \to G_0$ is an $G_\bullet$-equivariant map. Such a vector bundle is called a $G_\bullet$-representation \cite{Mackenzie2005}.
Define
\[
    F_E \;:\; G_\bullet^{\mathrm{op}} \to \mathsf{Vect}, 
    \qquad 
    F_E(p) = E_p, 
    \qquad 
    F_E(g) = \hat{g}^{-1} : E_{t(g)} \to E_{s(g)}.
\]
Functoriality of $F_E$ follows immediately from $\widehat{hg} = \hat{h} \circ \hat{g}$.
In this case a natural bias $b$ is simply a $G_\bullet$-invariant section of $E$.

\begin{lemma}
\label{prop:equiv-naturality}
Let $E, E' \to G_0$ be $G_\bullet$-representations with associated feature functors 
$F_{E}$, $F_{E'}$ as above. A family of linear maps 
$\Phi = \{\Phi_p : E_{p} \to E_{p}'\}_{p \in G_0}$  is $G_\bullet$-equivariant (i.e. is a morphism of $G_\bullet$-representations) if and 
only if it defines a natural transformation $\Phi : F_{E} \Rightarrow F_{E'}$.
\end{lemma}

\begin{proof}
This is the statement that a morphism of $G_\bullet$-representations is a natural transformation of the action functors. We will repeat the proof here.
Let $\hat{g}_1 = \varrho_E(g), \hat{g}_2 = \varrho_{E'}(g)$. Naturality of $\Phi$ requires that for every arrow $g \in G_1$ with $s(g)=p$ and $t(g) = q$,
\[
    F_{E'}(g) \circ \Phi_{q} \;=\; \Phi_{p} \circ F_{E}(g),
\]
i.e., $\hat{g}_2^{-1} \circ \Phi_{q} = \Phi_{p}\circ \hat{g}_1^{-1}$. Composing on 
the right with $\hat{g}_1$ and on the left with $\hat{g}_2$ gives
\[
    \Phi_{q} \circ \hat{g}_1 \;=\; \hat{g}_2 \circ \Phi_{p}
    \qquad \forall\, g \in G_1,
\]
which is precisely $G_\bullet$-equivariance of $\{\Phi_p\}$.
\end{proof}

\subsection{General Case}
\label{subsec:GeneralCase}
Now we would like to reinterpret the previously defined convolution operations in terms of these categorical constructions. We will show that our constructions are a special case of the example in \cite[Appendix~B.1]{maruyama}.

\begin{thm}
\label{thm:conv-natural}
Suppose that $k$ is a Lie groupoid kernel for a LieGrpdConv layer (Definition~\ref{def:LieGrpdConv}). Then if $k$ has compact support on $X/\!/G_1$ and is $L^1$-integrable with respect to $\dd\lambda^{\rho(y)}(\tilde{g})\dd\mu([y])$, then the associated LieGrpdConv is a special case of a category convolution in the sense of \cite[Definition~5]{maruyama}, where the feature functors $F,F'$ are defined by trivial vector bundles $\mathbb{R}^{c_{\rm in}}\times G_1 \times (X/\!/G_1)$ and $\mathbb{R}^{c_{\rm out}}\times G_1 \times (X/\!/G_1)$ respectively.
\end{thm}

\begin{proof}
We will first complete the proof in the case where the bias is zero and then show that this trivially implies the general case.

We first note that the space $G_1 \times (X/\!/G_1)$ can itself be interpreted as a $G_\bullet$-space with anchor map $\rho:(g,[x]) \mapsto t(g)$, and action given by $a(h,(g,[x])) = (h\circ g,[x])$. The associated action groupoid is $G_1^{(2)}\times(X/\!/G_1)$ where $G_1^{(2)}$ is the space of compatible 2-tuples. We will denote this action groupoid by $H_\bullet$. An appropriate definition of a feature functor is then to take the base spaces to be $\Omega(g,[x]) = X/\!/G_1$, and take $E$ to be a $H_\bullet$-representation $E \to G_1 \times (X/\!/G_1)$, where $E$ is a trivial bundle with fiber $\mathbb{R}^{c_{\rm in}}$, and set $E_{(g,[x])}=E|_{(g,[x])}$. Then set,
\begin{align*}
    \tau_{(h,g,[x])} &= 1_{X/\!/G_1}\\
    \tilde{\tau}_{(h,g,[x])} &= 1_{X/\!/G_1}
\end{align*}
and let $L_{(h,g,[x])}= \varrho(h^{-1})$, where $\varrho$ is the representation morphism. Identically, we take $E'$ to be another $H_\bullet$-representation, where $E'$ is a trivial bundle with fiber $\mathbb{R}^{c_{\rm out}}$.
The arrow bundle is then given fiber-wise by $I(g,[y]) = t^{-1}(\rho(y))\times (X/\!/G_1)$. Next, a category kernel is a family of functions, $\{K_{(\tilde{g},[y])\to(g,[y])}\}$, with values $K_{(\tilde{g},[y])\to(g,[y])}((h,\tilde{g},[y]),[x])$. For this to be well-defined we require that $h\tilde{g} = g$, so that $h = g\tilde{g}^{-1}$. Also recall that a feature $f$ is a function $f : G_1\times X/\!/G_1 \to C^0(X/\!/G_1,\mathbb{R}^{c_{\rm in}})$. 
Applying \cite[Equation~3]{maruyama}, we have
\begin{align*}
    \mathcal{L}_K\{f\}_{(g,[y])}([x]) = \int_{I(g,[y])}K_{(\tilde{g},[y])\to(g,[y])}((h,\tilde{g},[y]),[x])(\varrho(\tilde{g}g^{-1})f|_{(g,[y])})|_{[x]}\dd\lambda^{\rho(y)}(\tilde{g})\dd\mu([y])
\end{align*}
Restricting to the case where each $f|_{(g,[y])}$ is an constant function, and so that the feature obeys $F(u)f|_{(g,[y])} = f|_{u^{-1}\circ (g,[y])}$ (i.e. is equivariant), we have
\begin{align*}
    \mathcal{L}_K\{f\}_{(g,[y])}([x]) = \int_{I(g,[y])}K_{(\tilde{g},[y])\to(g,[y])}((h,\tilde{g},[y]),[x])f|_{(\tilde{g},[y])}\dd\lambda^{\rho(y)}(\tilde{g})\dd\mu([y])
\end{align*}
where we can now identify
\[K_{(\tilde{g},[y])\to(g,[y])}((h,\tilde{g},[y]),[x]) = k(\tilde{g}^{-1}g,[x],[y])\]
where we note that $\tilde{g}^{-1}g = \tilde{g}^{-1}h\tilde{g}$. So, finally, we have
\begin{align*}
    \mathcal{L}_K\{f\}_{(g,[y])}([x]) = \int_{I(g,[y])}k(\tilde{g}^{-1}g,[x],[y])f|_{(\tilde{g},[y])}\dd\lambda^{\rho(y)}(\tilde{g})\dd\mu([y]) = C_k\{f\}(g,[x])
\end{align*}
as required. We notice that $\mathcal{L}_K\{f\}_{(g,[y])}([x])$ is independent of $[y]$, which is a consequence of our assumptions on the feature $f$.

The Carath\'eodory regularity condition and the $L^1$-integrability condition follows so long as $k$ is assumed to be an $L^1$-integrable function of compact support, and the integrated naturality condition follows simply from the left-invariance of the Haar system as mentioned in \cite[Appendix~B.1]{maruyama}. Finally, suppose that $b$ is a natural bias for $F'$ as in Definition \ref{def:naturalbias}. This means that $\varrho(\tilde{g}g^{-1})b(g,[x]) = b(\tilde{g},[x])$ for all $\tilde{g}$. It follows that $b$ is equivariant, and satisfies the requirements of a bias function as mentioned in Definition \ref{def:LieGrpdConv}. This completes the proof.
\end{proof}

\begin{thm}
Let $X_\bullet$ be the action groupoid for a $G_\bullet$-space $X$. Suppose that $k$ is a Lie groupoid kernel for a LieGrpdLiftingConv layer (Definition~\ref{def:LieGrpdLiftingConv}). 
Then if $k$ has compact support on $X/\!/G_1$ and is $\mathcal{L}^1$-integrable with respect to $\dd\lambda^{\rho(y)}(\tilde{g})\dd\mu([y])$, then the associated LieGrpdLiftingConv gives rise to a continuous natural transformation of feature functors $F_0 : X_\bullet^{\rm op}\to \mathsf{Vect}$ and $F'_0 : X_\bullet^{\rm op} \to \mathsf{Vect}$, defined as the whiskerings (or horizontal compositions) $F_0 = F\circ\mathcal{E}^{\rm op}$ and $F'_0 = F'\circ\mathcal{E}^{\rm op}$ of the feature functors of Theorem~\ref{thm:conv-natural} along the groupoid morphism $\mathcal{E} : X_\bullet \to H_\bullet$.
\end{thm}
\begin{proof}
    Consider the feature functors $F, F'$ described in the proof of Theorem \ref{thm:conv-natural}, and consider the morphism of groupoids $\mathcal{E}: X_\bullet \to H_\bullet$ given by $x \mapsto (1_x,[x])$ and $(g,x) \mapsto (g,1_x,[x])$. Then let $F_0 = F \circ \mathcal{E}^{\rm op}$ be the precomposition of $F$ with the opposite functor of $\mathcal{E}$, and similarly define $F'_0 = F'\circ \mathcal{E}^{\rm op}$. The precomposition allows us to replace $f(g,[x])$ in the definition of the LieGrpdLiftingConv with $f(g\cdot x)$ by identifying $g\cdot x$ with $(g,1_x,[x])\cdot (1_x,[x])$. It suffices to show that the family $(\mathcal{L}_{K})_{\mathcal{E}(x)}, x \in X$ defines a continuous natural transformation $F_0 \Rightarrow F'_0$. Indeed, the family $(\mathcal{L}_{K})_{\mathcal{E}(x)}, x \in X$ is called the whiskering, or horizontal composition, of the natural transformation $\mathcal{L}_K$ with $\mathcal{E}$ \cite{MacLane1978}. The only remaining thing is to show that it is continuous in the sense of \cite[Definition~3]{maruyama}. This follows immediately from the fact that the quotient map is continuous and the unit map $x \to 1_x$ is smooth. This completes the proof.
\end{proof}

\begin{thm}
Let $H_\bullet$ be the action groupoid of $G_\bullet$ on $G_1\times (X/\!/G_1)$ as in the preceding definitions.
The GrpdIGP (Definition~\ref{def:GrpdIGP}) defines a natural transformation from a feature functor $F : H_\bullet^{\rm op} \to \mathsf{Vect}$ to the constant feature functor $F_{\rm const} :H_\bullet^{\rm op} \to \mathsf{Vect}$ defined such that $F_{\rm const}(x) = C^0(H_0,\mathbb{R}^{\rm c_{final}})^{G_\bullet}\subset C^0(H_0,\mathbb{R}^{\rm c_{final}})$.
\end{thm}
\begin{proof}
    Consider the exact same setup as in the proof of Theorem \ref{thm:conv-natural}, except now we will define the kernel to be given by
    \[k(\tilde{g}^{-1}g,[x],[y]) = \delta_{[x]}\]
    where $\delta_{[x]}$ is the Dirac delta distribution supported on the orbit $[x]$. In other words, it is the Radon-Nikodym derivative of the Dirac delta measure with respect to the measure $\dd\mu([y])$. Given this, we have

    \begin{align*}
    \mathcal{L}_\delta\{f\}_{(g,[y])}([x]) &= b([x]) + \int_{I(g,[y])}\delta_{[x]}f|_{(\tilde{g},[y])}\dd\lambda^{\rho(y)}(\tilde{g})\dd\mu([y]) \\
    &= b([x]) + \int_{t^{-1}(\rho(x))} f|_{(\tilde{g},[x])}\dd\lambda^{\rho(x)}(\tilde{g})\\
    &= \textrm{GrpdIGP}\{f\}([x])
    \end{align*}
    where we have used the left-invariance of the Haar system in applying the sifting property to $\dd\lambda^{\rho(x)}$. 
\end{proof}

Taken together, the full architecture defines a diagram of natural transformations
\[
    F_{\mathrm{in},0} 
    \xRightarrow{\;\mathcal{L}_{K_0}\circ \mathcal{E}^{\rm op}\;}\; 
    F^{(c_1)} 
    \xRightarrow{\;\mathcal{L}_{K_1}\circ \mathcal{E}^{\rm op}\;}\; 
    \cdots 
    \xRightarrow{\;\mathcal{L}_{K_N}\circ \mathcal{E}^{\rm op}\;}
    F^{(c_N)} 
    \xRightarrow{\;\mathrm{GrpdIGP}\circ \mathcal{E}^{\rm op}\;} 
    F_{\mathrm{const},0}
\]
with each arrow a natural transformation. Here, $F^{(c_i)}$ are feature functors corresponding to intermediate layers of the neural network. We note that by lifting from $X_\bullet$ to $H_\bullet$ at the start, all successive natural transformations are essentially restricted to the image of the embedding $\mathcal{E}$ in $H_\bullet$, i.e. all feature functors are precomposed with $\mathcal{E}$. The 
$G_\bullet$-equivariance of the network as a whole is then a formal consequence of 
functoriality of composition of natural transformations. 
\begin{rem}
    We have not yet mentioned activation functions. As shown by Maruyama, any globally Lipschitz pointwise activation function can be used to define what is called a scalar-gated nonlinearity, which is a specific type of continuous natural transformation and thus does not break equivariance \cite[Appendix A Lemma 1]{maruyama}.
\end{rem}
\begin{rem}
    In this section we have not addressed the Lie algebroid parameterization introduced in section \ref{sec:algebroid_kernel}. A Lie algebroid is not readily interpreted as a category the same way that a groupoid is. However, as mentioned in the introduction, a groupoid may be interpreted as giving rise to a simplicial complex satisfying certain conditions that allow it to be interpreted using the language of higher homotopy theory and higher category theory. In particular a far-reaching generalization of Lie groupoids can be seen in Lie $n$-groupoids and Lie $\infty$-groupoids. Taking this perspective, amidst the folklore \cite{nlabalgebroid} (see Harel's 1980 article `On Folk Theorems' \cite{Harel1980} for a discussion of folklore in mathematics and computer science), there is a higher categorical interpretation of Lie algebroids as a special case of Lie $\infty$-algebroids \cite{Bruce2011}, which are supposedly Lie $\infty$-groupoids that are said to be only infinitesimally extended over their base space. That is, they can be seen as a special type of $\infty$-category, using the simplicial model of higher category theory. In particular, within this folklore interpretation, some authors argue that Lie ($1$-)algebroids should be seen as differential $\infty$-groupoids constructed from so-called infinitesimal simplices introduced by Kock \cite{Kock2009}. Recent peer-reviewed work instead presents them as simplicial objects in supermanifolds \cite{aintablian2025differentiable,AintablianThesis}. Therefore, it is unclear how to interpret Lie algebroid-equivariant convolutional neural networks in terms of Maruyama's category convolutions, and we leave this for future research.
\end{rem}

\section{Acknowledgements}
M. R. Astwood acknowledges funding as a PhD Candidate of the University of Manitoba Department of Mathematics via the GETS program, the University of Manitoba Graduate Fellowship, and financial support from supervisor Dr. Derek Krepski. M. R. Astwood also acknowledges helpful discussions with Fran\c{c}ois-Guillaume Lemesre (Senior Machine Learning Engineer, Descartes Systems Group).

\bibliography{bib}
\end{document}